\documentclass{amsart}
\usepackage{amsthm}
\usepackage{lmodern}
\usepackage[english]{babel}
\usepackage{amsmath}
\usepackage{amssymb}
\usepackage{mathptmx} 
\usepackage{color}
\usepackage{graphicx}
\usepackage{caption}
\usepackage{subcaption}
\usepackage{float}
\usepackage[all,cmtip]{xy}
\usepackage{bm}
\usepackage{enumerate}
\usepackage{pinlabel}
\newtheorem{theorem}{\bf Theorem}[section]
\newtheorem{lemma}[theorem]{\bf Lemma}
\newtheorem{definition}[theorem]{\bf Definition}
\newtheorem{corollary}[theorem]{\bf Corollary}
\newtheorem{proposition}[theorem]{\bf Proposition}

\newcommand{\rme}{\mathrm{e}}
\newcommand{\rmi}{\mathrm{i}}

\begin{document}

\title{Quasipositive links and electromagnetism}

\author{
Benjamin Bode}
\date{}

\address{Department of Mathematics, Graduate School of Science, Osaka University, Toyonaka, Osaka 560-0043, Japan}
\email{ben.bode.2013@my.bristol.ac.uk}




\maketitle
\begin{abstract}
For every link $L$ we construct a complex algebraic plane curve that intersects $S^3$ transversally in a link $\tilde{L}$ that contains $L$ as a sublink. This construction proves that every link $L$ is the sublink of a quasipositive link that is a satellite of the Hopf link. The explicit construction of the complex plane curve can be used to give upper bounds on its degree and to create arbitrarily knotted null lines in electromagnetic fields, sometimes referred to as vortex knots. Furthermore, these null lines are topologically stable for all time. We also show that the time evolution of electromagnetic fields as given by Bateman's construction and a choice of time-dependent stereographic projection can be understood as a continuous family of contactomorphisms with knotted field lines of the electric and magnetic fields corresponding to Legendrian knots.
\end{abstract}
%

\section{Introduction}\label{sec:intro}
This article studies links that are transverse intersections $f^{-1}(0)\cap S^3$ of the vanishing set of a complex polynomial $f:\mathbb{C}^2\to\mathbb{C}$ and the unit 3-sphere $S^3=\{(u,v)\in\mathbb{C}^2:|u|^2+|v|^2=1\}$.
It was shown by Rudolph \cite{rudolph83, rudolph84} and Boileau and Orevkov \cite{boileau} that the set of links that can arise in this way, the \textit{transverse} $\mathbb{C}$\textit{-links}, is identical to the set of quaispositive links.

\begin{definition}A link is called quasipositive if it is the closure of a braid $B$ of the form $B=\prod_{j=1}^\ell w_j \sigma_{i_j}w_j^{-1}$, where $w_j$ is any braid word and $\sigma_{i_j}$ is a positive standard generator of the braid group.
\end{definition}

An excellent overview of the topic can be found in \cite{rudolph}.

In \cite{bode:2016polynomial} the author and Dennis showed that every link can arise as a transverse intersection $f^{-1}(0)\cap S^3$, where $f:\mathbb{C}^2\to\mathbb{C}$ is a polynomial in complex variables $u$, $v$ and the complex conjugate $\overline{v}$. Since $f$ is holomorphic in $u$, but not in the second variable $v$, we call such polynomials \textit{semiholomorphic}. The proof in \cite{bode:2016polynomial} is constructive. It describes an algorithm that finds a corresponding semiholomorphic polynomial for any link $L$. Furthermore, the degree of the constructed polynomial can be bounded in terms of the number of strands and number of crossings of a braid that closes to $L$ \cite{bode:2016polynomial}.

Rudolph showed that every link is the sublink of a quasipositive link \cite{rudolph}. However, his proof does not allow us to find the corresponding polynomials. In Section \ref{sec:constr} we present a modification to the algorithm from \cite{bode:2016polynomial} that produces holomorphic polynomials $f$ with $L\subset f^{-1}(0)\cap S^3$ for any link $L$. This shows that every link arises as a subset of a quasipositive link (i.e., transverse $\mathbb{C}$-link) in a constructive way, which allows us to show that $f^{-1}(0)\cap S^3$ is a satellite of the Hopf link. Since the differences to the algorithm in \cite{bode:2016polynomial} are very small, the bounds on the degree of the polynomial remain almost unchanged and we can give a quite explicit description of $f^{-1}(0)\cap S^3$. We obtain the following result.
\begin{theorem}
\label{thm:main1}
Let $B$ be a braid on $s$ strands that closes to the link $L$. Then $L$ is the sublink of a quasipositive link $\tilde{L}=f^{-1}(0)\cap S^3$ for some complex polynomial $f:\mathbb{C}^2\to\mathbb{C}$. Furthermore, $\tilde{L}$ is a satellite of the Hopf link and $\deg_u f=s$.
\end{theorem}
An upper bound for the total degree of $f$ in terms of the number of strands and the length of $B$ is given in Section \ref{sec:constr} in Proposition \ref{prop:bound}. 

The constructive nature of the proof in Section \ref{sec:constr} has an application in the context of knotted fields in physics. We show in Section \ref{sec:EM} how the polynomials can be used to construct configurations of electromagnetic fields $\mathbf{F}_t=\mathbf{B}_t+\rmi\mathbf{E}_t:\mathbb{R}^3\to\mathbb{C}^3$, whose set of null lines $\mathbf{F}_t^{-1}(0,0,0)$ contains $L$ for all time $t\in\mathbb{R}$. The link type is thus a stable topological feature of the electromagnetic field. 
\begin{theorem}
\label{thm:main2}
For every link $L$ we can construct an electromagnetic field $\mathbf{F}_t$ that satisfies Maxwell's equations and such that the set of null lines $\mathbf{F}_t^{-1}(0,0,0)$ contains $L$ for all time $t$.
\end{theorem} 
The existence of such fields again follows from \cite{rudolph} (although it is not a result that appears in the literature), but the algorithm allows us to write down these fields explicitly.

We conclude with some remarks on the construction of electromagnetic fields with knotted field lines that remain knotted for all time. It turns out that in electromagnetic fields that are constructed using a method due to Bateman \cite{bateman} as in \cite{kedia1} knotted flow lines are Legendrian knots and time evolution is given by a continuous family of contactomorphism. This means that the problem of constructing electromagnetic fields with flow lines of a prescribed knot type can be reformulated in terms of holomorphic continuations of certain complex-valued functions.

\textbf{Acknowledgements:} The author is grateful to Mark Dennis, Sachiko Hamano, Daniel Peralta-Salas and Jonathan Robbins for valuable discussions and encouragements. The author would also like to express his gratitude towards John Keating for asking an excellent question that led to a lot of this work and towards Peter Feller for pointing out some of Rudolph's results. The author was supported by the Leverhulme Trust Research Programme Grant RP2013-K-009, SPOCK: Scientific Properties Of Complex Knots and by JSPS KAKENHI Grant Number JP18F18751 and a JSPS Postdoctoral Fellowship as JSPS International Research Fellow.

\section{Constructing transverse $\mathbb{C}$-links}
\label{sec:constr}

Let $L$ be a link in $S^3$ and $B$ be a braid on $s$ strands that closes to $L$. In this section we describe how to construct a complex algebraic plane curve $f:\mathbb{C}^2\to\mathbb{C}$ such that $L$ is ambient isotopic to a subset of $f^{-1}(0)\cap S^3$. The first steps in this construction are completely identical to the ones in \cite{bode:2016polynomial}. Let $\mathfrak{C}$ denote the set of components of $L$ or equivalently the set of cycles of the permutation of $s$ points associated to $B$. Let $s_C$ denote the number of strands that belong to the component $C\in\mathfrak{C}$. 

The first thing we need to do in order to construct $f$ is to find a very particular parametrisation of $B$, namely a parametrisation of the form
\begin{equation}
\label{eq:trigpara}
\bigcup_{C\in\mathfrak{C}}\bigcup_{j=1}^{s_C}\left(X_{C,j}(t),Y_{C,j}(t),t\right):=\bigcup_{C\in\mathfrak{C}}\bigcup_{j=1}^{s_C}\left(F_C\left(\frac{t+2\pi j}{s_C}\right),G_C\left(\frac{t+2\pi j}{s_C}\right),t\right),\qquad t\in[0,2\pi],
\end{equation}
where $F_C$, $G_C:\mathbb{R}\to\mathbb{R}$ are trigonometric polynomials, i.e., 
\begin{equation}F_{C}(t)=\sum_{k=-N_{C}}^{N_{C}}a_{C,k}\rme^{\rmi kt}\text{  and  }G_{C}(t)=\sum_{k=-M_{C}}^{M_{C}}b_{C,k}\rme^{\rmi kt}
\label{eq:deftrig}
\end{equation}
with $a_{C,-k}=\overline{a}_{C,k}$ and $b_{C,-k}=\overline{b}_{C,k}$ for all $C\in\mathfrak{C}$ and all $k$. The natural numbers $N_C$ and $M_C$ are the degrees of $F_C$ and $G_C$, respectively.

We showed in \cite{bode:2016polynomial} how such a parametrisation can be found for any braid using trigonometric interpolation.

Once we have a parametrisation as in Equation (\ref{eq:trigpara}) we can define a family of polynomials $g_{\lambda}:\mathbb{C}\times[0,2\pi]\to\mathbb{C}$
\begin{equation}
g_{\lambda}(u,t)=\prod_{C\in\mathfrak{C}}\prod_{j=1}^{s_C}\left(u-\lambda(X_{C,j}(t)+Y_{C,j}(t))\right). 
\end{equation}
Then for all $\lambda>0$ the vanishing set $g_{\lambda}^{-1}(0)$ is the braid $B$ in the parametrisation in Equation (\ref{eq:trigpara}) scaled by $\lambda$. Furthermore, when we expand the product, $g_{\lambda}$ turns out to be a polynomial not only in the complex variable $u$, but also in $\rme^{\rmi t}$ and $\rme^{-\rmi t}$ \cite{bode:2016polynomial}.

If we substitute $v$ for $\rme^{\rmi t}$ and the complex conjugate $\overline{v}$ for $\rme^{-\rmi t}$ in the polynomial expression of $g_{\lambda}$, we obtain a family of polynomials $f_{\lambda}:\mathbb{C}^2\to\mathbb{C}$ in complex variables $u$, $v$ and $\overline{v}$. Note that $f_{\lambda}(u,\rme^{\rmi t})=g_{\lambda}(u,t)$. The proof of the following result can be found in \cite{bode:2016polynomial}.
\begin{theorem}[\cite{bode:2016polynomial}]
For small enough values of $\lambda>0$ the intersection $f_{\lambda}^{-1}\cap S^3$ is the closure of $B$ and hence $L$.
\end{theorem}

Now we consider a slightly different construction, where instead of replacing every instance of $\rme^{\rmi t}$ in $g_{\lambda}$ by $v$ and every instance of $\rme^{-\rmi t}$ by $\overline{v}$, we replace $\rme^{\rmi t}$ by $v$ and $\rme^{-\rmi t}$ by $\tfrac{1}{v}$. We call the resulting function $p_{\lambda}$. In contrast to $f_{\lambda}$, the function $p_{\lambda}$ is not a semiholomorphic polynomial, but a rational function, where the denominator is some power of $v$ and its numerator $\tilde{f}_{\lambda}$ is a polynomial in $u$ and $v$ and hence holomorphic. 

The largest exponent of $v$ and $\overline{v}$ in $f_{\lambda}$ is $\sum_{C\in\mathfrak{C}}\max\{N_C,M_C\}$ and the degree of the corresponding monomial with respect to $u$ is 0. 
It follows that $\tilde{f}_{\lambda}(u,v)\neq 0$ if $v=0$ and therefore the zero level set of 
$\tilde{f}_{\lambda}$ is identical to the zero level set of $p_{\lambda}$. We know that the intersection of the zero level set of $\tilde{f}_{\lambda}$ and $S^3$ is a transverse $\mathbb{C}$-link (or equivalently quasipositive) as long as the intersection is transverse (which it is for sufficiently small values of $\lambda$), since $\tilde{f}_{\lambda}$ itself is a complex polynomial. In particular, if we started with the parametrisation of a braid that does not close to a quasipositive link, then $p_{\lambda}^{-1}(0)\cap S^3$ is not the desired link $L$ even when $\lambda$ is small.

This might be a bit surprising at first, since $p_{\lambda}$ shares many crucial properties of $f_{\lambda}$, of which we know that $f_{\lambda}^{-1}(0)\cap S^3$ is the desired link. For example we have that $p_{\lambda}(u,\rme^{\rmi t})=f_{\lambda}(u,\rme^{\rmi t})=g_{\lambda}(u,t)$. Which part of the proof in \cite{bode:2016polynomial} is it then that does not work in this setting?


Before we move on, we should briefly recapitulate what can go wrong in the construction in \cite{bode:2016polynomial} if $\lambda$ is not small enough. For every $r\in[0,1]$ and $t\in[0,2\pi]$ the $s$ (not necessarily distinct) roots of the complex polynomial $f_{\lambda}(u,r\rme^{\rmi t})$ are denoted by $u_{\lambda,j}(r,t)$, $j=1,2,\ldots,s$, and for every fixed $t$ the roots $(u_{\lambda,j,t}(r),r\rme^{\rmi t})=(u_{\lambda,j}(r,t),r\rme^{\rmi t})$ form parametric curves in $\mathbb{C}^2$ parametrised by $r$. Note that $u_{\lambda,j,t}(r)=\lambda u_{1,j,t}(r)$.

There are three different problems that can occur.
\begin{enumerate}[(i)]
\item One such curve $(u_{\lambda,j,t}(r),r\rme^{\rmi t})$ does not intersect $S^3$ at all.
\item Both $(u_{\lambda,j,t}(r),r\rme^{\rmi t})$ and $(u_{\lambda,k,t}(r),r\rme^{\rmi t})$ with $j\neq k$ intersect $S^3$, say at values $r_1$ and $r_2$, but the images of $[r_1,1]$ under $r\mapsto u_{\lambda,j,t}(r)$ and of $[r_2,1]$ under $r\mapsto u_{\lambda,k,t}(r)$ are not disjoint in $\mathbb{C}$.
\item One curve $(u_{\lambda,j,t}(r),r\rme^{\rmi t})$ intersects $S^3$ more than once.
\end{enumerate}  
In the construction in \cite{bode:2016polynomial} all of these problems can be resolved by making $\lambda$ small. We can choose $\lambda$ small enough that for every $j$ and $t$ the modulus $|u_{\lambda,j,t}(r)|=\lambda|u_{1,j,t}(r)|$ satisfies $|u_{\lambda,j,t}(r)|^2+r^2<1$ for some $r\in(0,1)$. Since $|u_{\lambda,j,t}(1)|^2+1^2\geq1$, every curve $(u_{\lambda,j,t}(r),r\rme^{\rmi t})$ must therefore intersect $S^3$, possibly at $r=1$ (which resolves (i)). Moreover, if we choose $\lambda$ small enough, all such intersections occur at some $r\in(1-\delta,1]$ for any given $\delta>0$. This $\delta$ can be chosen such that (ii) and (iii) are also avoided. The details of this can be found in \cite{bode:2016polynomial}, but all that is really used is that distinct roots of a polynomial are smooth functions of its coefficients. We now prove our main result.

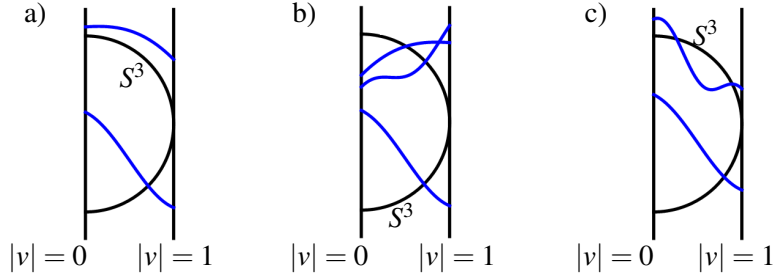
\begin{figure}
\centering
\labellist
\large
\pinlabel a) at 30 200
\pinlabel $S^3$ at 110 150
\pinlabel $|v|=0$ at 40 0
\pinlabel $|v|=1$ at 145 0
\endlabellist
\includegraphics[height=3.5cm]{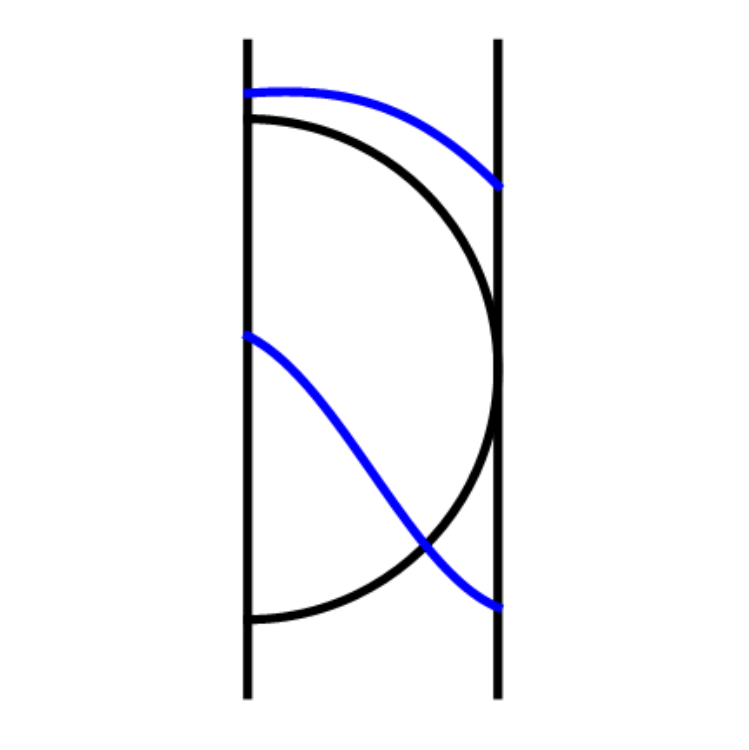}
\labellist
\large
\pinlabel b) at 30 200
\pinlabel $S^3$ at 110 35
\pinlabel $|v|=0$ at 50 0
\pinlabel $|v|=1$ at 155 0
\endlabellist
\includegraphics[height=3.5cm]{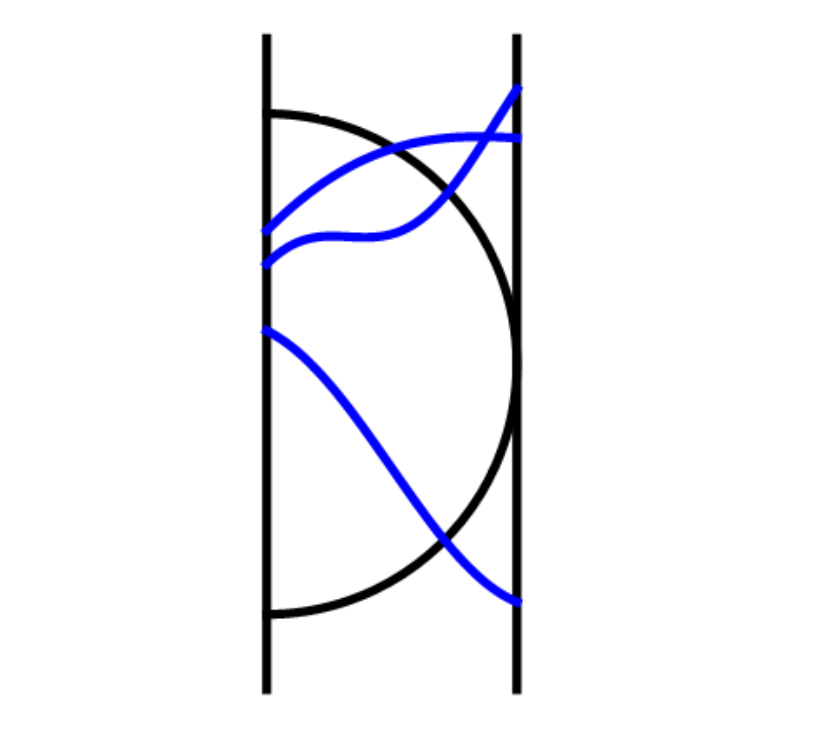}
\labellist
\large
\pinlabel c) at 30 200
\pinlabel $S^3$ at 120 185
\pinlabel $|v|=0$ at 45 0
\pinlabel $|v|=1$ at 145 0
\endlabellist
\includegraphics[height=3.5cm]{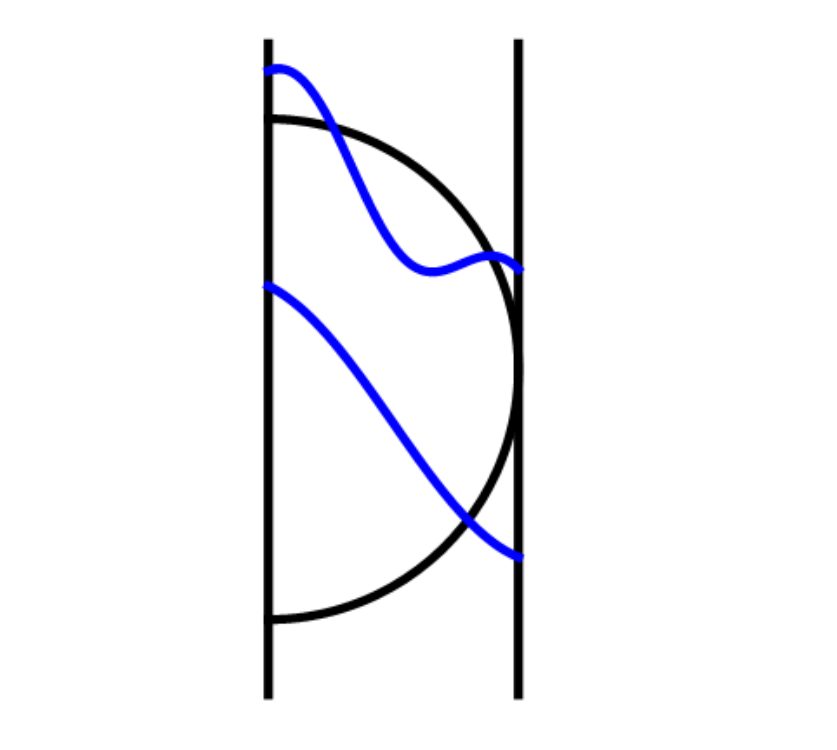}
\caption{2-dimensional sketches of the three different problems that can arise when $\lambda$ is not small enough. The vertical lines represent the planes in $\mathbb{C}^2$ where $|v|=0$ and $|v|=1$ respectively and a fixed value of $\arg v=t$. The half circle is the intersection of $S^3$ with $\arg v=t$ for a fixed $t$. The blue curves show the curves $u_{\lambda,j,t}(r)$. a) One of the curves $u_{\lambda,j,t}(r)$ does not intersect $S^3$. b) The $u$-coordinate of two curves coincide between their intersection points with $S^3$ and with $|v|=1$. c) One of the curves $u_{\lambda,j,t}(r)$ intersects $S^3$ multiple times.}
\end{figure}

\begin{figure}
\centering
\labellist
\large
\pinlabel $S^3$ at 120 185
\pinlabel $|v|=0$ at 35 0
\pinlabel $|v|=1$ at 145 0
\endlabellist
\includegraphics[height=3.5cm]{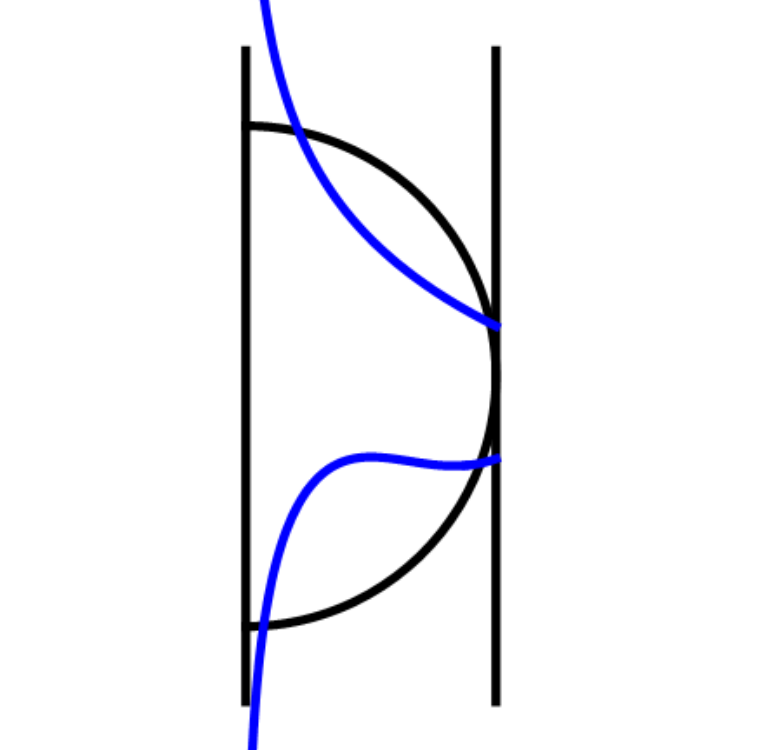}
\caption{The $u$-coordinates of the roots of $\tilde{f}_{a,b}(u,re^{it})$ go to infinity as $r$ goes to zero.}
\end{figure}

\begin{theorem}
\label{sublink}
For every link $L$ there exists a complex polynomial $\tilde{f}:\mathbb{C}^2\to\mathbb{C}$ such that $L$ is a sublink of the link formed by the transverse intersection of $\tilde{f}^{-1}(0)\cap S^3$, which is a satellite of the Hopf link. 
\end{theorem}

\begin{proof}
The first two problems (i) and (ii) can be resolved by choosing $\lambda$ small enough by exactly the same arguments as in \cite{bode:2016polynomial}.

The key difference between $f_{\lambda}$ and $\tilde{f}_{\lambda}$ is that $\tilde{f}_{\lambda}(u,0)$ is simply a constant (not depending on $u$). This means that $|u_{j,1,t}(r)|$ goes to infinity as $r$ goes to zero from above. Thus the third problem above can not be resolved by choosing $\lambda$ small enough as sketched above. No matter how small $\lambda$ is chosen, we can never achieve that all intersection points as above occur with $r\in(1-\delta,1]$ for a small enough $\delta$. However, as in \cite{bode:2016polynomial} small values of $\lambda$ imply that every curve $(u_{\lambda,j,t}(r),r\rme^{\rmi t})$ intersects $S^3$ in the interval $r\in(1-\delta,1)$ and inside this interval the interscetions are unique for each $j$ and $t$.

The modulus $|u_{1,j,t}(r)|$ of the roots is bounded on the compact set $r\in[\epsilon,1-\delta]$, where $\epsilon>0$ is chosen small. For small enough values of $\lambda$ there are therefore no intersections of $(u_{\lambda,j,t}(r),r\rme^{\rmi t})$ and $S^3$ with $r\in[\epsilon,1-\delta]$. Thus for $r=1$ the points $(u_{\lambda,j,t}(1),\rme^{\rmi t})$ are not in the open unit 4-ball in $\mathbb{R}^4$. As $r$ decreases each curve intersects $S^3$ in a unique point with $r\in[1-\delta,1]$. For $r\in[\epsilon,1-\delta]$ there are no further intersections, so all curves remain in the unit 4-ball. Since $|u_{j,1,t}(r)|$ goes to infinity as $r$ goes to zero, every curve must intersect $S^3$ at least one more time with $r\in(0,\epsilon]$. By the same arguments as in \cite{bode:2016polynomial} this intersection is unique in $(0,\epsilon]$ for all $j$ and $t$ if $\lambda$ is chosen small enough.

Hence for small enough $\lambda$ and every $t$ and $j$ the curve $(u_{j,\lambda,t}(r),r\rme^{\rmi t})$ intersects $S^3$ in exactly two points, one very close to $r=1$ and one close to $r=0$. Thus the link $\tilde{L}:=\tilde{f}_{\lambda}^{-1}(0)\cap S^3$ is the union of two links $L_1$ (the intersection points near $r=1$) and $L_0$ (the intersection points near $r=0$).

Everything else in the proof in \cite{bode:2016polynomial} still holds, which means that $L_1$ is the desired link $L$, the closure of the braid $B$, whose parametrisation we started with. The construction offers a concrete description of how $L$ is linked with the rest of $\tilde{L}$ (or with the notation of above how $L_1$ links with $L_0$). For small $\lambda$ the link $L_1$ lies in a tubular neighbourhood of $|v|=1$ in $S^{3}$ and $L_0$ lies in a tubular neighbourhood of $|u|=1$ in $S^3$. Therefore $L_1$ and $L_0$ each lie inside a solid torus and these two tori form a Hopf link. By construction $\deg_u f_{\lambda}=s$, which concludes the proof.  
\end{proof}

\begin{corollary}
\label{quasi}
For every link $L$ there exists a link $\tilde{L}$ such that $L$ is a sublink of $\tilde{L}$ and $\tilde{L}$ is a quasipositive satellite of the Hopf link.
\end{corollary}

This corollary follows from the theorem, since by Boileau and Orevkov \cite{boileau} $f^{-1}(0)\cap S^3$ is a quasipositive link for every complex polynomial $f$.

Rudolph explains in \cite{rudolph} that every link $L$ is the sublink of a quasipositive link $\tilde{L}$, which is a satellite of $L$. Theorem \ref{sublink} and Corollary \ref{quasi} first of all add another possibility for the quasipositive link $\tilde{L}$, which in our construction is not a satellite of $L$, but a satellite of the Hopf link. Perhaps more importantly, the constructive nature of the proof allows us to give a bound on the polynomial degree and use the resulting functions in the context of electromagnetic fields (cf. Section \ref{sec:EM}).

The degrees of the semiholomorphic polynomials $f_{\lambda}$ that are constructed in \cite{bode:2016polynomial} can be bounded in terms of the number of strands $s$ and the number of crossings $\ell$ of the parametrised braid. This is because the degrees of the polynomial are related to the degrees $M_C$ and $N_C$ of the trigonometric polynomials $F_C$ and $G_C$, that parametrise the $X$- and $Y$-coordinate, respectively. Since these functions are found by trigonometric interpolation their degree is determined by the number of data points used for the interpolation, which in turn depends on the number of strands and crossings of the chosen braid. Since the relevant steps of the construction of $\tilde{f}$ are identical to those in the construction of $f_{\lambda}$, these imply the following result.

\begin{proposition}[cf. Lemma 4.1, Lemma 4.2 and Corollary 4.3  in \cite{bode:2016polynomial}]
\label{prop:bound}
Let $B$ be a braid on $s$ strands with $\ell$ crossings and let $L$ be its closure. Then there exists a complex polynomial $\tilde{f}:\mathbb{C}^2\to\mathbb{C}$ such that $L$ is a sublink of the link formed by the transverse intersection of $f^{-1}(0)$ and the unit three-sphere $S^3$, $\deg_u \tilde{f}=s$,
\begin{equation}
\deg_v \tilde{f}=2\sum_{C\in\mathfrak{C}}\max\{N_C,M_C\}\leq 2\sum_{C\in\mathfrak{C}}\left
\lfloor\frac{(s_C+1)(\ell s_C-1)+\ell s_C(s-s_C)}{2}\right\rfloor,
\end{equation}
and
\begin{align}
\deg\tilde{f}\leq &\sum_{C\in\mathfrak{C}}\max\left\{\left
\lfloor\frac{(s_C+1)(\ell s_C-1)+\ell s_C(s-s_C)}{2}\right\rfloor,s_C\right\}\nonumber\\
&+\sum_{C\in\mathfrak{C}}\left
\lfloor\frac{(s_C+1)(\ell s_C-1)+\ell s_C(s-s_C)}{2}\right\rfloor,
\end{align}
where as before $\mathfrak{C}$ denotes the set of components of $L$ and for every $C\in\mathfrak{C}$ the number of strands that $C$ consists of is denoted by $s_C$.
\end{proposition}

In the case of a knot, there is only one component so that the bound simplifies to
\begin{equation}
\deg\tilde{f}\leq 
2\left\lfloor\frac{(s+1)(\ell s-1)}{2}\right\rfloor.
\end{equation}

Also note that it follows immediately from the construction that the number of components of $\tilde{L}$ that are not part of $L$ is bounded above by the number of strands of $B$.

\noindent\textbf{Example:} Consider the figure-eight knot, the closure of the braid $\sigma_1\sigma_2^{-1}\sigma_1\sigma_2^{-1}$ parametrised by 
\begin{equation}
\bigcup_{j=1}^3\left(\cos\left(\frac{2t+2\pi j}{3}\right),\sin\left(\frac{2(2t+2\pi j)}{3}\right),t\right), \qquad t\in[0,2\pi].
\end{equation}
Note that the figure-eight knot is not quasipositive \cite{rudolph:quasi}.
The corresponding braid polynomial is
\begin{align}
g_{\lambda}(u,t)&=\prod_{j=1}^3\left(u-\lambda\left(\cos\left(\frac{2t+2\pi j}{3}\right)+\rmi\sin\left(\frac{2(2t+2\pi j)}{3}\right)\right)\right)\nonumber\\
&=u^3-\frac{3}{4}u\lambda^2(\rme^{2\rmi t}-\rme^{-2\rmi t})-\frac{1}{8}(4\lambda^3(\rme^{2\rmi t}+\rme^{-2\rmi t})+\lambda^3(\rme^{4\rmi t}-\rme^{-4\rmi t})).
\end{align}
Replacing each $\rme^{\rmi t}$ by $v$ and every $\rme^{-\rmi t}$ by $\tfrac{1}{v}$ we obtain
\begin{align}
p_{\lambda}(u,v)&=u^3-\frac{3}{4}u \lambda^2(v^2-\frac{1}{v^2})-\frac{1}{8}(4\lambda^3(v^2+\frac{1}{v^2})+\lambda^3(v^4-\frac{1}{v^4})\nonumber\\
&=\frac{1}{8v^4}\left(8u^3v^4-6u\lambda^2(v^6-v^2)-4\lambda^3(v^6+v^2)-\lambda^3(v^8-1)\right).
\end{align}
Therefore
\begin{equation}
\tilde{f}_{\lambda}(u,v)=(8u^3v^4-6u\lambda^2(v^6-v^2)-4\lambda^3(v^6+v^2)-\lambda^3(v^8-1).
\end{equation}
Figure \ref{fig:holo} shows the zero level set of $\tilde{f}_{1/3}$ on the unit three-sphere after projecting it into $\mathbb{R}^3$. One component is the figure-eight knot. The other components lie close to the $z$-axis and close near the point at infinity. Note that the number of strands of the components that are not the figure-eight knot (in this case 4) could be larger than the number of strands of the constructed knot (in this case 3), but the number of extra components (in this case 2) is bounded above by the number of strands.


As pointed out in \cite{bode:2016polynomial} the bound on the degree is not very tight. The upper bound from Proposition \ref{prop:bound} is 44, while the actual degree is 8.

\begin{figure}
\centering
\labellist
\large
\pinlabel a) at 0 400
\endlabellist
\includegraphics[height=5.5cm]{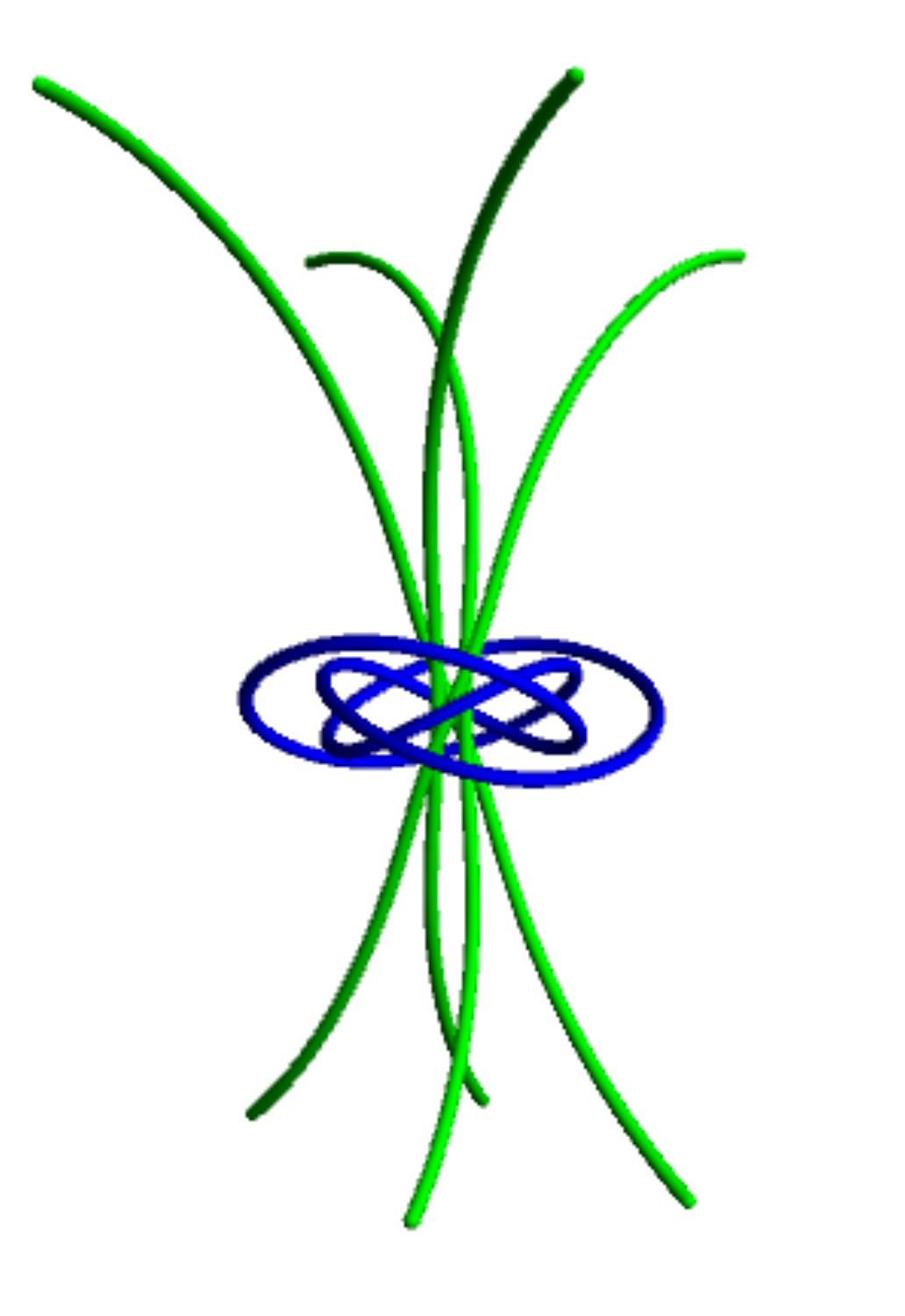}\qquad
\labellist
\large
\pinlabel b) at 50 340
\endlabellist
\includegraphics[height=4.5cm]{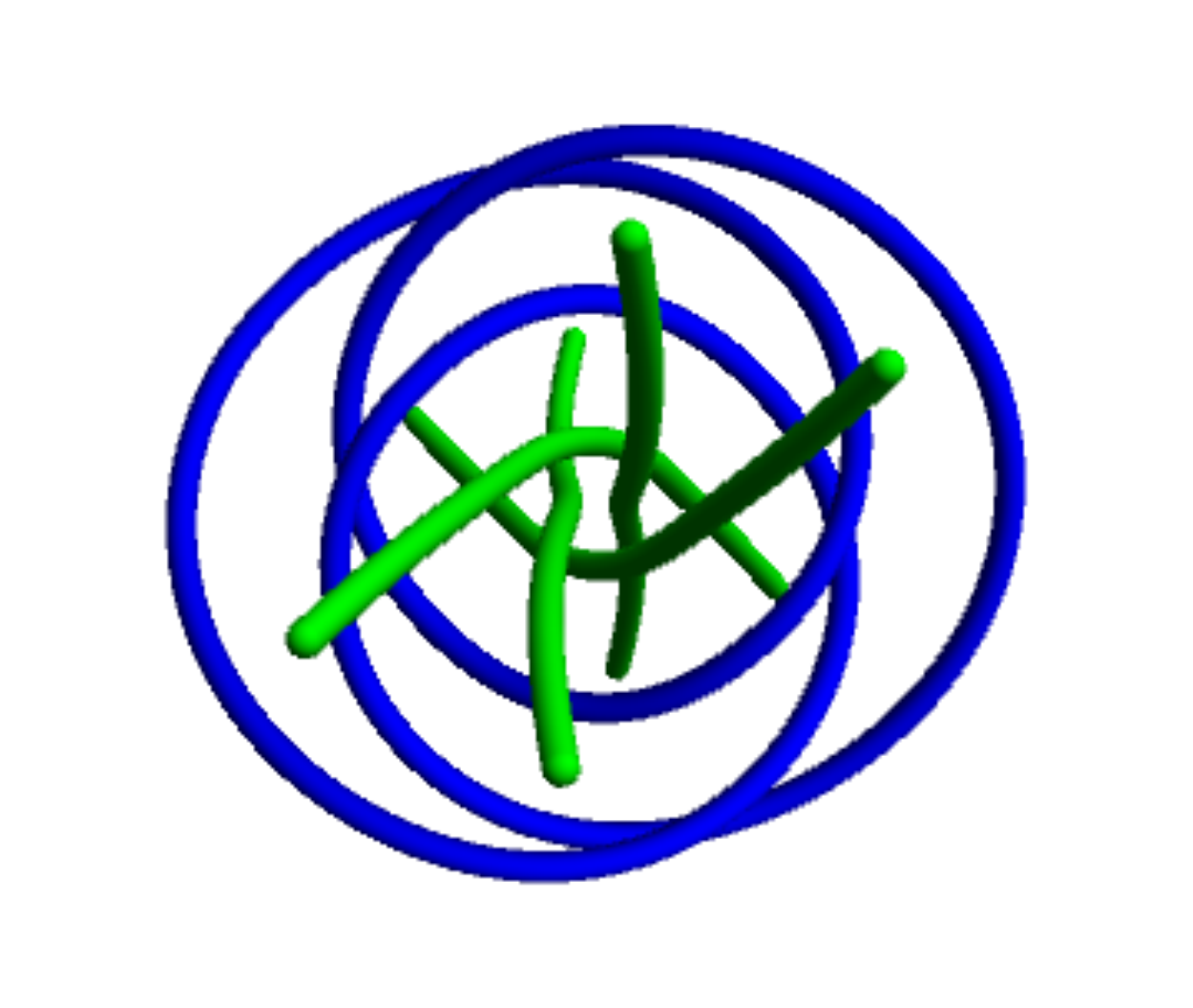}
\caption{The zero level set of the complex polynomial $\tilde{f}_{1/3}$ on the unit three-sphere projected into $\mathbb{R}^3$. The figure-eight knot (in blue) lies in a solid torus, whose core is the unit circle in the $z=0$-plane. The other components (in green) lie close to the $z$-axis. \label{fig:holo}}
\end{figure}

\section{Knotted null lines in Maxwell's equations}
\label{sec:EM}

Recent years have seen a growing interest in configurations of physical systems that contain knots and links. Knots have been found in very diverse areas such as quantum mechanics \cite{berry:2001hydrogen}, optics \cite{bd:2001knotted, mark}, non-linear field theories in particle  physics \cite{sutcliffe}, topological fluid dynamics \cite{daniel1, daniel2, lr:2012jones, moffatt:1969degree},  and liquid crystals \cite{km:2016topology, ma:2014knotted, ma:2016global}.
In scalar complex fields knots arise as the nodal set of the field, while systems that are described by 3d vector fields can contain knotted flow lines or vortex knots.
There are two main approaches to the study of knotted fields: the statistical analysis of the knots that occur in these fields naturally \cite{td} and the constructive approach, which aims to construct such field configurations for given knot types \cite{mark}.
The remarkable thing is that despite the great variety of different physical systems, many of the constructions above are in one way or another based on finding polynomial maps from 3-dimensional space to the complex numbers whose vanishing set is a given knot, such as the complex plane curves for algebraic links studied by Brauner \cite{brauner} and Milnor \cite{milnor} composed with a stereographic projection. This was the original motivation for the construction in \cite{bode:2016polynomial}, which offers us the maximal freedom in the choice of link we can construct.

If the system is not taken to be static, these maps can be used as initial conditions for the time evolution of the system.
The dynamics of the system is then governed by a differential equation or a given energy functional that needs to be minimised. As the system evolves, so does the knot. Typically, intersections occur, parts of the knotted curve pass through each other and the knot type changes and potentially disappears altogether.
While in most physical systems we do not expect the fact that the polynomials constructed in Section \ref{sec:constr} are holomorphic to have any effect on the time evolution of the knot or link, the case of electromagnetism is different. Here the Cauchy-Riemann equations can be linked to Maxwell's equations, which determine the time evolution of the field.

An electromagnetic field $\mathbf{F}_t$ is a family of vector fields $\mathbf{F}_t=\mathbf{E}_t+\rmi \mathbf{B}_t:\mathbb{R}^3\to\mathbb{C}^3$, where $t$ is the time parameter, $\mathbf{E}_t:\mathbb{R}^3\to\mathbb{R}^3$ is the time-dependent electric field and $\mathbf{B}_t:\mathbb{R}^3\to\mathbb{R}^3$ is the time-dependent magnetic field. Together they satisfy Maxwell's equations in free space:
\begin{align}
\nabla \cdot \mathbf{B}_t&=0,\label{eq:maxwella}\\
\nabla \times \mathbf{E}_t+\partial_{t} \mathbf{B}_t &=0,\label{eq:maxwellb}\\
\nabla \cdot \mathbf{D}_t&=0,\label{eq:maxwellc}\\
\nabla \times \mathbf{H}_t-\partial_{t} \mathbf{D}_t&=0,\label{eq:maxwelld}
\end{align} 
where $\mathbf{H}_t=\frac{1}{\mu_0}\mathbf{B}_t$, $\mathbf{D}_t=\tfrac{1}{\varepsilon_{0}}\mathbf{E}_t$, $\mu_0$ and $\varepsilon_0$ are the electric and magnetic permittivity respectively and and $\nabla$ is the gradient with respect to the three spatial coordinates.

There are different ways in which an electromagnetic field can be knotted. Links can be formed by the flow lines (field lines) of the electric field $\mathbf{E}_t$ and the magnetic field $\mathbf{B}_t$. Furthermore, the null lines of the electromagnetic field $\mathbf{F}_t^{-1}(0,0,0)$ can be knotted or linked. Sometimes these links are referred to as {\it{vortex knots}} \cite{bouwmeester}, but we cannot always guarantee that the vector fields twist around all lines with $\mathbf{F}_t^{-1}(0,0,0)$ like they do for vortex knots in fluids. Instead the null lines should be thought of as the vortex lines of an electromagnetic phase function as in \cite{phasesing}, where these lines are called RS vortices.

In this section we show how the holormophicity of the polynomials in Section \ref{sec:constr} leads to time-dependent complex-valued vector fields $\mathbf{F}_t$ that satisfy Maxwell's equations and have RS vortex lines $\mathbf{F}_t^{-1}(0,0,0)$ that contain the desired link $L$ for all time $t$. The link type $L$ is thus a stable topological feature of the electromagnetic field. 

Connection between complex polynomials and electromagnetic fields are due to Bateman \cite{bateman} and Ra{\~n}ada \cite{ranada} and employed in \cite{kedia1}, where Kedia et al. give concrete examples of constructions of electromagnetic fields with knotted field lines. In their work the field lines take the shape of torus knots and links.  

They consider the following time-dependent inverse stereographic projection from $\mathbb{R}^3\cup\{\infty\}$ to $S^{3}$
\begin{equation}
\label{eq:stereo3}
u=\frac{x^2+y^2+z^2-t^2-1+2\rmi z}{x^2+y^2+z^2-(t-\rmi)^2}\qquad \text{and}\qquad v=\frac{2(x-\rmi y)}{x^2+y^2+z^2-(t-\rmi)^2},
\end{equation}
where $x$, $y$ and $z$ are the three spatial coordinates and $t$ represents time. We can thus write the electromagnetic field in terms of $u$ and $v$.


They go on to show that the topology of the field lines of 
\begin{equation}
\label{eq:electromag}
\mathbf{F}(u,v)=\nabla f(u,v)\times \nabla g(u,v)
\end{equation}
does not change with time, where $f$ and $g$ are arbitrary holomorphic functions in $u$ and $v$. They also point out that for $f=u^p$ and $g=v^q$ the resulting electric and magnetic fields both contain field lines that form the $(p,q)$-torus link. Hence there is a construction of flow lines in the shape of torus links that are stable for all time. It should also be noted that the mathematics is completely equivalent to that of the time evolution of unbreakable filaments in fluid flows \cite{kedia1, kedia2}. The holomorphicity of $f$ and $g$ ensures that the constructed electromagnetic fields are null for all time $(\mathbf{E}_t\cdot \mathbf{B}_t=0)$. Then the field lines of $\mathbf{E}_t$ and $\mathbf{B}_t$ evolve like filaments in a fluid that are pushed in the direction of the Poynting vector (parallel to $\mathbf{E}_t\times\mathbf{B}_t$).

In these configurations field lines lie on nested tori around the torus links. Therefore any other periodic flow line must give a cable of torus link. 
A more general construction of knotted flow lines is still an open and very challenging problem. Other constructions have been suggested in recent years \cite{hoyos}, but to our knowledge no knot outside the family of torus links (or possible cables of these) has been constructed. Kedia et al. found conditions that guarantee that initial configurations $\mathbf{F}_0$ have topologically stable flow lines \cite{kedia2}, but so far this has not led to more constructions.

The field $\mathbf{F}_t$ in Equation (\ref{eq:electromag}) can be rewritten as
\begin{equation}
\label{eq:h}
\mathbf{F}_t(u,v)=h(u,v) (\nabla u \times \nabla v),
\end{equation}
where $h(u,v)=\partial_u f \partial_v g-\partial_v f \partial_u g$. From this expression, which is given in \cite{kedia1}, it is not too hard to see that the null lines of $\mathbf{F}_t$ are given by $h^{-1}(0)$, since $\nabla u \times \nabla v$ is never equal to $(0,0,0)$. Furthermore, the link type of $\mathbf{F}_{t}^{-1}(0,0,0)=h^{-1}(0)$ does not change over time.

We can thus construct electromagnetic fields that have topologically stable null lines of link type $L$ by finding complex polynomials (holomorphic in $u$ and $v$) that have $L$ as their nodal set on $S^3$. This was pointed out for the case of algebraic links by de Klerk et al. \cite{bouwmeester} by using the projection maps $(\varepsilon u, \varepsilon v):\mathbb{R}^3\to S_{\varepsilon}^3$ for some small $\varepsilon>0$ and the corresponding Milnor polynomials. However, neither the isolated singularity of the corresponding polynomials nor the property of having the desired link as the nodal set on 3-spheres of small radii $\varepsilon$ is required. The correct set of links to consider is therefore not the set of algebraic links, but rather Rudolph's transverse $\mathbb{C}$-links. It follows from Theorem \ref{sublink} and Corollary \ref{quasi} that every link is the subset of a stable vortex link of an electromagnetic field, which is summarised in Theorem \ref{thm:main2}.

Section \ref{sec:constr} provides us with an explicit construction of the required polynomials $h:\mathbb{C}^2\to\mathbb{C}$. We simply set $h=\tilde{f}$ and obtain the electromagnetic field from Equation (\ref{eq:h}).

Vortex knots haven also been studied in optical fields that are described by complex scalar fields $\Psi:\mathbb{R}^3\to\mathbb{C}$ satisfying the paraxial wave equation \cite{mark}. In certain regimes this is a good approximation to Maxwell's equations. In fact, the trefoil knot and the figure-eight knot have been realised experimentally in laser beams \cite{mark}.

It should however be noted that the electromagnetic fields that we construct as above should not be expected to be monochromatic, a shortcoming that they share with the fields of de Klerk et al \cite{bouwmeester}. Hence an experimental realisation of these fields as laser beams with methods as in \cite{mark} does not seem likely at the moment.

\section{Some remarks on knotted field lines in electromagnetic fields}
\label{sec:flow}

In this section we would like to point out some observations on the construction of knotted field lines in electromagnetic fields. To our knowledge there is no construction of stable knotted field lines, whose knot types are not torus knots or perhaps cables of torus knots (using the same fields) \cite{kedia1}. An overview on the topic of electromagnetic knots can be found in \cite{emreview}, which contains the relevant references.

Given that the fields $\mathbf{F}_t(u,v)=\nabla u^p\times\nabla v^q$ in \cite{kedia1}, whose electric and magnetic fields have field lines in the shape of $(p,q)$-torus links, are clearly inspired by Brauner's and Milnor's polynomial $u^p-v^q$, which intersects $S^3$ in a $(p,q)$-torus link, one might hope that the construction of complex polynomials in Section \ref{sec:constr} could be useful for the construction of knotted field lines too.

If this is the case, it is certainly not straightforward. It is important to keep in mind that the $(p,q)$-torus link that can be found as a closed field line in $\mathbf{F}_t(u,v)$ is actually the mirror image of the $(p,q)$-torus link given by the intersection of the nodal set of $u^p-v^q$ with $S^3$. For $F_t(u,v)=\nabla f\times\nabla g$ the functions $\text{Re}(fg)$ and $\text{Im}(fg)$ are constant along the field lines of the magnetic field and the electric field, respectively. The important feature of the knotted field lines therefore seems not to be that it is related to (the mirror image of) the nodal set of a holomorphic function on $S^3$, but rather that it is the preimage set of a locally extremal value of the functions $\text{Re}(fg)$ and $\text{Im}(fg)$.
All of this is explained in far more detail in \cite{kedia1}. The important point here is that even though Brauner's polynomial inspired the fields in \cite{kedia1}, we should not be too optimistic about new configurations with knotted field lines coming from the construction in Section \ref{sec:constr}.

However, there is another approach to the construction of knotted field lines in electromagnetism and again it is closely related to holomorphic functions as it also employs Bateman's construction. Recall that the fields in Bateman's construction take the form 
\begin{equation}
\mathbf{F}_t(u,v)=h(u,v)\nabla u\times\nabla v.
\end{equation}
The projection maps $u$ and $v$ are given by Equation (\ref{eq:stereo3}) and we are looking for the holomorphic function $h$. The key observation is that for our choice of $u$ and $v$ the vector fields $\text{Re}(\nabla u \times \nabla v)$ and $\text{Im}(\nabla u \times \nabla v)$, the part that is the same for all of the fields, span a plane field $\xi_t$ that defines a contact structure on $\mathbb{R}^3$ for every $t$. This means that $\text{Re}(\nabla u \times \nabla v)$ and $\text{Im}(\nabla u \times \nabla v)$ are linearly independent over $\mathbb{R}$ for all points in $\mathbb{R}^3$ and all time $t$ and
\begin{equation}
\alpha_t\wedge\text{d}\alpha_t\neq0,
\end{equation}
where $\alpha_t$ is the 1-form such that for all time $t$
\begin{equation}
\text{ker }\alpha_t=\xi_t=span(\{\text{Re}(\nabla u \times \nabla v),\text{Im}(\nabla u \times \nabla v)\}).
\end{equation}

In fact, $(\text{Re}(\nabla u \times \nabla v)$ and $\text{Im}(\nabla u \times \nabla v)$ are always perpendicular everywhere and $\alpha_t\wedge\text{d}\alpha_t$ is given by
\begin{equation}
\alpha_t\wedge\text{d}\alpha_t=1024\frac{(1+t^2+x^2+y^2-2tz+z^2)^3}{(t^4-2t^2(x^2+y^2+z^2-1)+(x^2+y^2+z^2+1)^2)^6},
\end{equation}
which never vanishes. Note that we have a contact structure for all values of $t$, so the time evolution as dictated by Maxwell's equations is actually a 1-parameter family of contactomorphisms.

\begin{lemma}
\label{lem:leg}
A knotted field line of the electric or the magnetic field of a field constructed with Bateman's construction using the projection maps $u$ and $v$ as in Eq. (\ref{eq:stereo3}) is Legendrian with respect to the corresponding contact structure $\xi_t$ for all $t$.
\end{lemma}

\begin{proof}
Suppose that $L$ is a link that is formed by field lines of the electric field. Then at each point $(x,y,z)\in\mathbb{R}^3$ on $L$ the tangent vector at that point $T(x,y,z)$ is in the span of $\{\text{Re}(\nabla u \times \nabla v),\text{Im}(\nabla u \times \nabla v)\}$ because
\begin{equation}
T(x,y,z)=Re(h(u,v)\nabla u\times\nabla v)=\text{Re}(h(u,v))\text{Re}(\nabla  u\times\nabla v)-\text{Im}(h(u,v))\text{Im}(\nabla u\times\nabla v).
\end{equation}
Similarly, links that are given by flow lines of the magnetic fields are in the span of $\{\text{Re}(\nabla u \times \nabla v),\text{Im}(\nabla u \times \nabla v)\}$ because
\begin{equation}
T(x,y,z)=Im(h(u,v)\nabla u\times\nabla v)=\text{Re}(h(u,v))\text{Im}(\nabla  u\times\nabla v)+\text{Re}(h(u,v))\text{Im}(\nabla u\times\nabla v).
\end{equation}

In other words, every link that is formed by electric or magnetic field lines is a Legendrian link with respect to the contact form $\xi_t$. Furthermore, this is true for all time $t$. The time evolution given by the $t$-dependence of $u$ and $v$ is via contactomorphisms carrying Legendrian links to Legendrian links.
\end{proof}

We would like to use these facts for the construction of knotted field lines. For every contact structure on $\mathbb{R}^3$ every knot $K$ can be realised as a Legendrian knot with respect to that contact structure in infinitely many ways \cite{etnyre}. For each of these parametrisations $\phi:S^1\to \mathbb{R}^3$ of $K=\phi(S^1)$ we obtain a complex-valued function $h:K\to\mathbb{C}$ given by
\begin{align}
\text{Re}(h(x,y,z))&=\nabla \phi(x,y,z)\cdot\text{Re}(\nabla u\times\nabla v),\qquad\text{ for all }(x,y,z)\in K,\nonumber\\
\text{Im}(h(x,y,z))&=\nabla \phi(x,y,z)\cdot\text{Im}(\nabla u\times\nabla v),\qquad\text{ for all }(x,y,z)\in K.
\end{align} 

We defined $h$ as a function on a subset $K$ of $\mathbb{R}^3$, but by identifying points in $\mathbb{R}^3\cup\{\infty\}$ with points in $S^3$ via the functions $u$ and $v$, we can think of $h$ as a function on a subset $K$ of $S^3\subset\mathbb{C}^2$. The decisive question is whether among the infinitely many Legendrian parametrisations of $K$ there is one parametrisation such that the complex-valued function $h$ can be extended to a function $\tilde{h}:U\to\mathbb{C}$, where $U$ is an open set in $\mathbb{C}^2$ containing $S^3$, such that $\tilde{h}$ is holomorphic on $S^3$, i.e.,
\begin{equation}
\partial_{\overline{z_1}}\tilde{h}(z_1,z_2)=0\qquad\text{ and }\qquad \partial_{\overline{z_2}}\tilde{h}(z_1,z_2)=0
\qquad\text{ for all }(z_1,z_2)\in S^3\subset\mathbb{C}^2.
\end{equation}

In summary, we have the following result.
\begin{proposition}
A knot $K$ can be realised as a topologically stable field of an electric or magnetic field using Bateman's construction and the choice of projection maps $(u,v)$ as in Eq. (\ref{eq:stereo3}) if and only if there exists a parametrisation of $K$ that is Legendrian with respect to $\xi_t$ and such that $h:K\to\mathbb{C}$ can be extended to some $\tilde{h}:U\to\mathbb{C}$ such that $\tilde{h}$ is holomorphic on $S^3$, where $U$ is an open set in $\mathbb{C}^2$ containing $S^3$.
\end{proposition}

It is not clear at all if such parametrisations exist for every knot or if perhaps the links constructed in \cite{kedia1} remain the only ones. We would like to point out that Ra{\~n}ada's construction has also been interpreted using contact geometry \cite{contact}. Furthermore, Costa e Silva, Goulart and Ottoni interpreted knotted solutions to Maxwell's equations as foliations of space time \cite{foliation}. 

These different approaches to known constructions of topologically stable knotted field lines in electromagnetic fields could lead to new constructions and hopefully, eventually, to a classification of the links that can arise as stable field lines.




\end{document}